\documentclass[12pt]{amsart}
\usepackage{amsmath,amsthm,epsfig,amssymb,amscd,amsfonts} 
\usepackage[all]{xy}
\usepackage{young}
\newtheorem{thm}[subsection]{Theorem}

\newtheorem{prop}[subsection]{Proposition}

\newtheorem{lem}[subsection]{Lemma}
\newtheorem{corol}[subsection]{Corollary}
\newtheorem{rem}[subsection]{Remark}
\theoremstyle{definition}
\newtheorem{Def}[subsection]{Definition}
\newtheorem{Not}[subsection]{Notation}

\newtheorem{Remark}[subsection]{Remark}

\newtheorem{proposition-definition}[subsection]{Proposition-Definition}

\begin{normalsize} \end{normalsize}

\newcommand{\codim}{\operatorname{codim}}

\newcommand{\CC}{{\mathbb C}}

\newcommand{\ZZ}{{\mathbb Z}}
\newcommand{\QQ}{{\mathbb Q}}

\newcommand{\NN}{{\mathbb N}}

\newcommand{\OOO}{{\mathcal O}}

\newcommand{\LLL}{{\mathcal L}}

\newcommand{\SSS}{{\mathcal S}}

\numberwithin{equation}{section}

\author{F. Laytimi}

\address {F. L.: Math\'ematiques - b\^{a}t. M2, Universit\'e Lille 1,
F-59655 Villeneuve d'Ascq Cedex, France}

\email {fatima.laytimi@math.univ-lille1.fr}

\author{W. Nahm}

\address{W. N.: Dublin Institute for Advanced Studies,
10 Burlington Road, Dublin 4, Ireland}

\email{wnahm@stp.dias.ie}


\title{ Ampleness equivalence and dominance for vector bundles }

\begin{document}

\date{}

\begin{abstract} Hartshorne in "Ample vector bundles" proved that $E$ is ample if and only if $\OOO_{P(E)}(1)$  is ample.
Here we generalize this result to  flag manifolds associated to a vector bundle $E$ on a complex manifold $X$:
For a partition $a$ 
we show that the line bundle $\it Q_a^s$ on the corresponding flag manifold 
$\mathcal{F}l_s(E)$
is ample if and only if $ \SSS_aE $ is ample.  In particular
$\det Q$ on $\it{G}_r(E)$ is ample if and only if  $\wedge ^rE$ is ample.\\ 
We give also a proof of the Ampleness Dominance theorem that does not depend on the saturation property of
the Littlewood-Richardson semigroup.
\end{abstract}

\maketitle

\section{Introduction} \setcounter{page}{1}
Let  $E$ be  a complex vector bundle  of rank $d $  on a compact complex manifold $X$
and $s$ be a sequence of integers $s=(s_0,s_1,\ldots s_m)$ such that \\ $0=s_0<s_1<\ldots<s_m=d.$
For $x\in X$ and the corresponding fiber $V=E_x,$ consider the manifold $\mathcal {F}l_s(V)$ of incomplete flags:
$$(V=V_{s_0}\supset V_{s_1}\supset\ \ldots \supset V_{s_m}=\{0\})= F,$$
 where $V_{s_i}$ is a vector subspace of $V,\ \codim  V_{s_j}=s_j.$
When $x$ varies, the $\mathcal {F}l_s(E_x)$ together form a manifold $\mathcal {F}l_s(E).$
 
 For a partition $a=(a_1,a_2, \ldots  )$ of length $l \leq d$  such that
\begin{align}
\nonumber a_1&=a_2  =\ldots =a_{s_1} \\
\nonumber a_{s_1+1 } &= a_{s_1+2}=\ldots=a_{s_2}\\
\nonumber a_{s_2+1 } &= a_{s_2+2}=\ldots=a_{s_3}\\
\nonumber \ldots &
\end{align}
there is a corresponding line bundle $\it Q_a^s$ over 
$\mathcal {F}l_s(E).$ When restricted to $\mathcal {F}l_s(V)$ 
it has the form
 $$ \it Q_a^s= \it  Q^{a_{s_1}}_{s_1}\otimes \it Q^{a_{s_2}}_{s_2}\otimes \ldots \otimes \it Q^{a_{s_m}}_{s_m} $$
where
$$ \it {Q_{s_j}}{|_F}=\det(V_{s_{j-1}}/V_{s_{j}}),\ \  1\leq j\leq m.$$

By Bott's formula \cite{ Bott}, for any $m\geq 0$ 
\begin{equation}\label{eq1}
  \pi_*(( \it Q_a^s)^m) \simeq \SSS_{ma}E 
\end{equation}  
\begin{equation}\label{eq2}
\it{R}^q{\pi_*}(( \it Q_a^s)^m) \simeq 0,\ \ {\mbox{if}}  \
 \ q>0
\end{equation}  
where  $\pi:\mathcal {F}l_s(E)\longrightarrow X$ is the natural projection
and $\SSS_a$ is the Schur functor corresponding to $a$.

Our Main Theorem is

\begin{thm} \label{main1}
     The line bundle $\it Q_a^s$ on  $\mathcal{F}l_s(E)$ is ample if and only if  the vector bundle $\SSS_{a}E$ on $X$  is ample.
\end{thm} 

A useful special case is
\begin{corol}
 The line bundle $\det Q$ on the grassmannian bundle 
$\it{G}_r(E)$ is ample if and only if  the vector bundle $\wedge ^rE$ is ample on $X. $
 \end{corol}

 Note that by Theorem \ref{main1}, the vanishing theorem of Demailly
 \cite{Dem} is valid under the minimal hypothesis  $\SSS_aE$  ample.

 Moreover we give a new and simple proof of the Ampleness Dominance result:

\begin{thm}\label{dom}
 Let $E$ be a vector bundle of rank $d$ and $a,b$ be partitions.   
If $b\preceq a,$ then  $\SSS_aE$ ample implies
 $\SSS_b E$ ample.
\end{thm}

 \begin{corol}
$\SSS_aE\otimes \SSS_bE$ is ample if and only if $\SSS_{a+b}E$ is ample.\\
For any positive integer $m$, $\SSS_aE$ is ample if and only if  $\SSS_{m a}E$ is ample.\\
 If $\Lambda^k E$ is ample and $s>0$ then  $\Lambda^{s+k} E$ is ample.
 \end{corol}

\section{Proof of Theorem 1.1.}

\medskip 
For the ``if''  direction of Theorem(\ref{main1}) 
we will use  the  following  lemma which  was first  pointed out to us  by L. Gruson  and which we have already used  in (\cite{IMRN}. lemma 2.6).
To our knowlege there is  no proof in the literature, so  we prove it here.
 
 \begin{lem}\label{Gruson}
 Let $G$ be a vector bundle on $Y$ and 
 $p: Y\longrightarrow X$ be a proper morphism satisfying: \\
 1) $ p_*G$ is an ample vector bundle,\\
 2) $G$ is ample along the fibers and \\
 3) the map  $p^*p_*G\longrightarrow G$ is surjective.\\
 Then $G$ is ample. 
\end{lem}
\begin{proof}
We will use the following result of Gieseker
(Lemma 2.1.page 101, \cite{Gi}):\\
Suppose $Y$ is proper over a field $k$  and  $E$ is a bundle over $Y$ generated by  its  global  sections. Then $E$  is  ample  if  and  only  if every 
quotient line bundle of $E_{|C}$ is  ample for   every curve $C$  in  $Y.$

Now we will prove Lemma(\ref{Gruson}):

Since $p_*G$ is ample, for large $m, \ \SSS^m(p_* G) $ is generated by global sections. By $3),$  
$\SSS^m(p^*p_* G)\longrightarrow  \SSS^m G$ is surjective.
Hence   $\SSS^m G$ is generated by sections.\\
Let $C$ be any curve in $Y.$ 

 If $C \subset p^{-1}(x)$ for some $x\in X, $ then $\SSS^mG|_{C}$ is ample by 2). This implies $G|_C$ ample. 
 
 If $C$ is not contained in any fiber, then $$\pi=p_{|C}: C\longrightarrow p(C)=C_1$$
 is a finite morphism. 
By $3)$ we have a surjective map
\begin{equation}\label{eq0}
i^*p^*p_*G \longrightarrow i^*G \longrightarrow 0 .
\end{equation}

The commutative diagram 
$$
\begin{array}{rcl}
C& \stackrel{i}{\hookrightarrow} & Y\\
{\pi}{\downarrow}& &\downarrow{p}\\
C_1&\stackrel{i_1}{\hookrightarrow}& X
\end{array}
$$
gives 
$$\SSS^m(i^*p^*(p_*G))=\pi^*i_1^*(\SSS^m(p_*G)). $$ 
Now $\pi^*i_1^*(\SSS^m(p_*G))$ is ample by assupmption 1).  Hence by (\ref{eq0})  $i^*G= G|_C$ is ample. This shows that $G$ is ample by Gieseker's result 
stated at the beginning of the proof.
\end{proof}

\medskip

{\bf I. Proof of the ``if''  direction of Theorem 1.1.} 

We apply Lemma(\ref{Gruson}) with $Y=\mathcal{F}l_s(E), \ \ G=\it Q_a^s $ 
and $p= \pi. $ By using equation(1.1), it is clear that 
all assumptions in Lemma(\ref{Gruson})  are satisfied. 
Hence $\it Q_a^s$ is ample. 
\medskip 

{\bf II. Proof of the 'only if' direction of Theorem 1.1.}
\medskip 

Some preparations:

\begin{thm}\label{useful}
 Let $\pi : Y  \longrightarrow X$ be a submersion between 
 two complex manifolds and $L$ be an ample line bundle on $Y.$ 
 Then there  exists $r_0$ such that 
 $\forall r \geq  r_0, \ \ \pi_*(L^r)$ is ample.
\end{thm} 
\begin{proof}
 We use the following result due to Mourougane (Theorem 1, \cite{M}):
 
If $\pi:Y \longrightarrow X$ is a submersion between two 
complex manifolds and $L$ is an ample  line bundle on $Y,$ 
then $\pi_*(L\otimes K_{Y/X})$ is ample or zero.
 
On the other hand, ampleness of $L$ implies that there exists $r_0,$ 
so that $\forall \ r \geq  r_0,  \  L^r\otimes K_{Y/X}^{-1}$ is ample. 

 Thus  $\pi_*( K_{Y/X}\otimes(L^r\otimes  K_{Y/X}^{-1}))=\pi_*(L^r)$ 
 is ample by Mourougane's result. 
\end{proof}
Now we come to the proof of the 'only if' direction of Theorem(\ref{main1}).

If  $\it Q_a^s$ is ample by Theorem(\ref{useful}) there  exists $r_0$ such that\\
 $\forall r \geq  r_0, \ \ \pi_*((Q_a^s)^r)$ is ample. By equation (\ref{eq1})
 $$\pi_*((Q_a^s)^r)\simeq \SSS_{ra}E. $$
 
Ampleness of  $\SSS_{a}E$ is deduced by the Ampleness Dominance\\ Theorem(\ref{dom}), since 
 $$ra\simeq a$$ in the dominance partial order.

\begin{rem}
 The proof of the Ampleness Dominance result (Theorem 3.7 p.175 in \cite{LP}) 
uses the saturation property of the Littlewood-Richardson semigroup for $r$ summands, $r$ arbitrary (p. 172 in \cite{LP}).
This property had been proven by Knutson and Tao \cite{KT} for $r=2$. We assumed that the extension to $r$ arbitrary is immediate,
but this appears not to be the case. A new and simpler proof that does not use saturation at all will be given in the next section.
\end{rem}

\section{ Dominance partial order on partitions and ampleness}

\begin{Def}
A partition $a=(a_1,a_2,\ldots a_d)$ of length $d$ is a non-increasing sequence of non-negative  integers $a_i. $   
The  weight $ \sum_{i=1}^{d} a_i$ is  denoted by $|a|.$

The dominance partial ordering of partitions is defined in 
\cite{Fulton} by:\\
Let $a=(a_1,a_2,\ldots a_d) $ and $ b=(b_1,b_2,\ldots b_d)$ be two partitions  of the same weight.\\ 
Then \ \ $a \preceq b$ \ \ if
\begin{align}
\nonumber a_1&\leq b_1 \\
\nonumber a_1+a_2&\leq b_1+b_2\\
\nonumber \ldots&\leq\ldots\\
\nonumber a_1+\ldots +a_{d-1}&\leq b_1+\ldots +b_{d-1}.\\
\nonumber a_1+\ldots +a_{d-1}+a_d&= b_1+\ldots +b_{d-1}+b_d.
\end{align}

  We extend this definition to non-increasing sequences of non-negative rational numbers:
$$a \preceq b\ \ {\mbox{if}}\ \ n_1a \preceq n_2b \ \ {\mbox{for some }}\ \   n_1,n_2 \in \NN, $$
 where \ $n_1a$ \ and \ $ n_2b$ are partitions of the same weight.
\end{Def}

When $a $ and $b$ are non-zero partitions of not necessarily equal weight an equivalent formulation is:
 $$a \preceq b \ \ \ \ \  {\mbox{if} } \ \  \ \ \ |b|\ a \ \preceq |a| \ \ b.$$
 
 If $ a \preceq b \ \ {\mbox{and}}\ \ b \preceq a,$ then we say $a$ is 
 equivalent to $b$ and write $a \simeq b.$
  
 \begin{Not}\label{not3}
For finite sequences $b=(b_1,\ldots,b_s)$ and  $c=(c_1,\ldots c_t)$ we set  
$$b\vee c  = (b_1,\ldots,b_s,c_1,\ldots c_t).$$
For a given positive integer $d$ let 
$$\LLL(d)= \{(l_1,\ldots,l_r) \ | \ r,\ l_i\in \NN ,\  l_1+l_2+\ldots+l_r=d  \}. $$

If $L=(l_1,\ldots,l_r)\in \LLL(d)$ and $a \in\QQ_{\geq 0}^d,$ then for $1\leq i\leq r,$
$a(L,i)$ is a sequence
of length $l_i,$ such that $a=a(L,1)\vee a(L,2)\ldots\vee a(L,r).$
\end{Not} 
  
\begin{Not}\label{not2}
Let $a,b$ be partitions of length $d$. If $V$ is a vector space of dimension $d$
we denote by
 $  a \otimes  b $\ \ \  the set of all partitions appearing in the decomposition of the tensor product $\ \ \SSS_ aV \otimes \SSS_ b V $ 
 as a direct sum of irreducible representations of $GL(V)$, and analogously for tensor products with more than two factors.\\
 For ease of understanding we write the partition 
  $\bf{1_d}$ instead of $\det, $ where $${\bf{1}}_{i}=(\underbrace{1,1,\ldots, 1}_{i\ \ times} ) ,$$ and ${\bf{1_i}}\vee{\bf{0_{d-i}}}$  instead of 
 $\wedge^i, $  where \ \  $${\bf{0_i}}= (\underbrace{0,à,\ldots, 0}_{i\ \ times} ). $$
\end{Not}

The only properties of the Littlewood-Richardson rules for the tensor product of irreducible representions we need to use in the sequel are the following.

\begin{prop}\label{dom1}
 Let  $a,b,c$ be partitions of length $d$.    \\ 
 If $ c\in a\otimes b $,
 then $c\preceq (a+b).$
\end{prop}

\begin{prop}\label{ZZ}
Let  $a,b,c,d, e, f$  be partitions of length $d$. If \\
 $c  \in  a\otimes b $ and  $f  \in d \otimes e,  $
then \ \ $  (c+f)  \in (a+d) \otimes ( b+e ) .$
 \end{prop}
 
 \begin{prop}\label{V}
  Let  $a,b,c$  be partitions of length $d$  and \\ 
  $L=(l_1,l_2, \ldots, l_r)\in \LLL(d). $   If \\
 $  c(L,i) \in a(L,i) \otimes b(L,i)  $ for $i=1,\ldots, r, $ then \\
 $c \in a \otimes  b $.
  \end{prop}
 
These propositions follow immediately from the rules of 
Littlewood-Richardson.For some background see Zelevinsky \cite{Z}.

\begin{Remark}\label{Rem} Since $d\in a\otimes b\otimes c $ is equivalent to the existence
of a partition $e$ with $e\in a\otimes b $ and $d\in e\otimes c $ the three properties
generalize immediately to tensor products with more than two factors.
\end{Remark}

The partitions $c $ such that $c\preceq  a , $ 
are the integral points of a rational cone $C(a)$. More precisely:

\begin{Def}
  Let $a$ be a partition. Let
 $$C(a)=\{\ b=(b,\ldots, b_d)\in\QQ_{\geq 0}^d \ \vert\ b_1\geq b_2\geq \ldots\geq b_d\geq 0  \ \ {\mbox{and } }\ \ b\preceq a\}.$$
\end{Def}

\begin{lem} \label{cone}
 Let $a=(a_1,\ldots,a_d)$ be a partition.  The cone $C(a)$ is generated by the set of sequences  \ \ 
 $\{v(L,a)\ \vert\ L\in \LLL(d)\},$ where
 $$v(L,a)=\frac{|a(L,1)|}{l_1} {\bf{1}}_{l_1}\vee \ldots \vee \frac{|a(L,r)|}{l_r} {\bf{1}}_{l_r}.$$
 \end{lem}

\begin{proof}
Consider the hyperplanes 
$$\Sigma(n)=\{b\in \QQ_{\geq 0}^d\ \vert\ \sum_{k=1}^n b_k = \sum_{k=1}^n a_k\},$$
$$D(n) = \{b\in \QQ_{\geq 0}^d\ \vert\ b_n=b_{n+1}\},$$
$n=1,\ldots, d$, where $b_{d+1}=0$.

$P(a)=C(a)\cap \Sigma(d)$ is a convex polytope in $\Sigma(d)$. Its vertices are given by an intersection of $d$ 
of the $2d$ hyperplanes just introduced, with $\Sigma(d)$ included among them.

 We will  prove the lemma by induction on $d$. For $d=1$ the claim is obvious. 
For a given vertex $v, $ consider the case where  $v\in\Sigma(n)$ for some $n\in \{1,\ldots,d-1\}$. One has
$$\Sigma(n)\cap C(a) \simeq P((a_1,\ldots,a_n))\times P((a_{n+1},\ldots,a_d)),$$
since $(b_1,\ldots,b_n)\in P((a_1,\ldots,a_n))$ implies \  $b_n\geq a_n$ and \\ $(b_{n+1},\ldots,b_d)\in P((a_{n+1},\ldots,a_d))$
implies $a_{n+1}\geq b_{n+1}$. The vertices of $P((a_1,\ldots,a_n))$ and $P((a_{n+1},\ldots,a_d))$ are known by induction.
If $v\in D(d), $ then $v\in \Sigma(d-1)$, too. The only remaining possibility is
$v=\Sigma(d)\cap D(1)\cap D(2)\cap\ldots\cap D(d-1)$, in which case the vertex is $v(L,a)$ for $L=(d)$.
\end{proof}

\begin{Not}
 For $l\in\NN$ \  let $\mu(l)={\it lcm}\ \{1,2,\ldots,l\}$.
\end{Not}

\begin{lem} \label{finite} Let $a$ be a partition.
 There is a finite set $\sigma(a)$ of partitions  such that if
  $b\in C(a)\cap\ZZ_{\geq 0}^d,$ then $b$ can be written as  $$b=c + \displaystyle{\sum_{L\in \LLL(d)} m_L\  v(L,\mu(l)a)}, \ \ c\in\sigma(a)$$ and $m_L\in\ZZ_{\geq 0}$ \ for all $L\in \LLL(d)$.
\end{lem}

\begin{proof}
Note that each $v(L,\mu(d)a)$ is a partition. According to Lemma \ref{cone}
 one has $$b= \displaystyle{\sum_{L\in \LLL(d)} \tau _L \ v(L,\mu(d)a)}$$ with
  $ \tau_L \in\QQ_{\geq 0}$ for all $L\in \LLL(d)$. The set of partitions $b$ 
 with $0\leq \tau_L<1$ for all $L\in \LLL(d)$ is a bounded subset of $\ZZ_{\geq 0}^d,$ thus finite.
\end{proof}

\begin{lem} \label{vinc}
 Let $a$ be a partition. For any $L\in \LLL(d)$
 $${v(L,\mu(d)a)}\in a^{\otimes \mu(d)}.$$
\end{lem}

The proof of this lemma will be subdivided into three elementary steps. 
\begin{lem}  \label{Lambda}
 Let  $k,m,q$ be non-negative integers with $k \leq d$ and let $mk=dq+s$
 with $0\leq s < k.$  Then
  $$((\det)^q\otimes  \wedge ^s ) \in ( \wedge ^k)^{\otimes m } . $$ 
 
 \end{lem} 
 
 \begin{proof}
 By induction on $m,$ with $m=1$ as trivial case. One has
$$ ((\det)^q\otimes  (\wedge ^s \otimes \wedge ^k)) \in ( \wedge ^k)^{\otimes (m+1) } $$ and 
$$\wedge ^{s+k}\in (\wedge ^s \otimes \wedge ^k)$$ for $s+k<d$,
$$ (\det \otimes \wedge^{s+k-d})\in (\wedge ^s \otimes \wedge ^k) $$ 
for $s+k\geq d$.
 \end{proof}
 
\begin{lem} \label{det} 
For any partition $a$ of weight ${|a|}$ and length $d$,
$$(\det)^ {|a|}\in {a}^{\otimes d}.$$
 \end{lem}
 
\begin{proof}
Let $\tilde a=( t_1,  t_2, \ldots )$ be the transpose of $a$. 
Then $a= \wedge^{t _1}+\wedge^{t_2}+ \ldots . $
By Lemma(\ref{Lambda})  the result is true for $\wedge^k$, $k=1,2,\ldots.$ 
By Proposition(\ref{ZZ})  it is true for $a$.
\end{proof}

The proof of  
Lemma(\ref{vinc}) follows immediately from Proposition(\ref{V}) and 
Lemma(\ref{det}).

\section{ Proof of Theorem \ref{dom}}

For the proof we need to recall a particular case of  Lemma(3.3) in \cite{semiample}.

\begin{lem}\label{semiample}
A vector bundle $E$ on $X$ is ample if and only if the following condition is true:
 Given any coherent sheaf $\mathcal{F}$ on $X$, there exists $N(\mathcal{F})\in\NN$  such that for any  
 $n\geq N(\mathcal{F})$ one has
 $$ H^q(X,E^{\otimes n} \otimes \mathcal{F})=0 \  \  {\mbox{for }} \  q>0. $$
\end{lem}

 Assuming $\SSS_a E$ ample and  \  $b \preceq a$ \ we want to 
 prove that $\SSS_bE$ is ample.

According to Lemma \ref{semiample} we have a map $N_a$ from the coherent sheaves on $X$ to $\NN$ such that
$$ H^q(X,(\SSS_aE)^{\otimes n} \otimes \mathcal{F})=0 \  \  {\mbox{for }} \  q>0\  \  {\mbox{and}} \  n>N_a(\mathcal{F}).$$
We want to prove the analogous property for $\SSS_bE$.\\
For any partition $c$ contained in $b^{\otimes m}$, according to Lemma \ref{dom1} and \ref{finite}
  one has a decomposition
$c = f+g$, where $f\in \sigma(a)$ and
$g= \displaystyle{\sum_{L\in \LLL(d)} \tau_L\  v(L, \mu(d)\ a)}, $ 
with $ \tau_L$  a non-negative integer  for all $ L\in \LLL(d).$ \\ 
 According to Proposition \ref{ZZ} and Lemma \ref{vinc}
 we have $g \in a^{\otimes \mu(d)s}$, where $s =\displaystyle{\sum_{L\in \LLL(d)}  \tau_L}$. 
 Define a map $N_b$ from the coherent sheaves on $X$ to $\NN$ by
 $$N_b(\mathcal{F})=[{\it max}_{f\in\sigma(a)}\frac {|f|+N_a(\SSS_fE\otimes \mathcal{F})|a|}{|b|}],$$
where the symbol $[ \ ]$ is the integral part.
 
 By Proposition(\ref{ZZ}) 
 $ H^q(X,(\SSS_bE)^{\otimes m} \otimes \mathcal{F})$ is a direct sum of subgroups of the cohomology groups
 $H^q(X,(\SSS_aE)^{\otimes  \mu(d)s}\otimes (\SSS_fE\otimes \mathcal{F}))$ with $f\in\sigma(a)$, $m|b|=|f|+\mu(d)|a|s$.
The latter groups vanish for  $m\geq N_b(\mathcal{F})$. Thus $\SSS_bE$ is ample. This finishes the proof of  Theorem \ref{dom}.

{\it Acknowledgement: The first author would like to thank  
Dublin  Institute for Advanced Studies for its hospitality. We would also 
like to thank Nagaraj D.S. for useful discussion.}

  \end{document}